\documentclass[a4paper,11pt,leqno,english]{smfart}
\usepackage{aeguill}
\usepackage{enumerate}
\usepackage{amssymb,amsmath,latexsym,amsthm}
\usepackage[T1]{fontenc}
\usepackage{smfthm}      
\usepackage{geometry}   
\usepackage{url}      
\usepackage[frenchb, english]{babel}
\usepackage[utf8]{inputenc}   
\usepackage{mathrsfs}
\usepackage{relsize}
\usepackage{xcolor}
\usepackage[all]{xy}
\usepackage{comment}
\definecolor{violet}{rgb}{0.0,0.2,0.7}
\definecolor{rouge2}{rgb}{0.8,0.0,0.2}
\usepackage{hyperref}
\usepackage{mathrsfs}
\hypersetup{
    bookmarks=true,         
    unicode=false,          
    pdftoolbar=true,        
    pdfmenubar=true,        
    pdffitwindow=false,     
    pdfstartview={FitH},    
    pdftitle={},    
    pdfauthor={},     
    colorlinks=true,       
   linkcolor=violet,          
    citecolor=violet,        
    filecolor=black,      
    urlcolor=cyan}           
\newcommand{\R}{\mathbb{R}}
\newcommand{\CC}{\mathbb{C}}

\newcommand{\Z}{\mathbb{Z}}
\newcommand{\N}{\mathbb{N}}

\newcommand{\D}{\mathbb D}

\newcommand{\Cx}{\mathbb{C}^*\times \mathbb{C}^{n-1}}

\renewcommand{\d}{\partial}

\newcommand{\vp}{\varphi}

\renewcommand{\O}{\mathcal{O}}
\newcommand{\Ox}{\mathcal{O}_{X}}
\newcommand{\ep}{\varepsilon}

\newcommand{\la}{\langle}

\newcommand{\ra}{\rangle}

\renewcommand{\ge}{\geqslant}
\renewcommand{\le}{\leqslant}

\newcommand{\Ric}{\mathrm{Ric} \,}
\newcommand{\om}{\omega}
\newcommand{\omvp}{\omega_{\varphi}}

\renewcommand{\D}{\mathbb D}

\newcommand{\PSH}{\mathrm{PSH}}
\newcommand{\Gb}{\mathcal G_{\beta}}
\newcommand{\Eb}{\mathcal E_{\beta}}
\newcommand{\Lb}{\mathcal L_{\beta}}
\newcommand{\pb}{\psi_{\beta}}
\newcommand{\Fb}{F_{\beta}}
\newcommand{\Hb}{H_{\beta}}

\newcommand{\tr}{\mathrm{tr}}

\newcommand{\ddc}{dd^c}
\newcommand{\vpb}{\varphi_{\beta}}
\newcommand{\vpbb}{\bar \varphi_{\beta}}

\newcommand{\Cb}{C_{\beta}}
\newcommand{\Vb}{V_{\beta}}
\renewcommand{\sb}{|s|^{2\beta}}
\newcommand{\sbb}{|s|^{2(1-\beta)}}
\newcommand{\rij}{R_{i\bar j k \bar l}}
\newcommand{\gt}{\tilde{g}}
\newcommand{\tra}{\tr_{\omv}\omb}
\newcommand{\He}{H_{\ep}}
\newcommand{\Gln}{\mathrm{GL}(n,\mathbb C)}

\newcommand{\omb}{\omega_{\beta}}
\newcommand{\ombb}{\bar \omega_{\beta}}
\newcommand{\ombm}{\omega_{\beta, \mathrm{mod}}}

\newcommand{\ome}{\omega_{\rm eucl}}

\newcommand{\omv}{\omega_{\vp_{\beta}}}
\newcommand{\cbe}{\chi_{\beta, \ep}}

\newtheorem*{cor}{Corollary}
\newtheorem*{thma}{Theorem A}
\newtheorem*{thmb}{Theorem B}
\newtheorem*{thmc}{Theorem C}

\numberwithin{equation}{section}

\begin{document}

\frontmatter 

\title[From cones to cusps]{Kähler-Einstein metrics: from cones to cusps}

\date{\today}
\author{Henri Guenancia}
\address{Department of Mathematics \\
Stony Brook University, Stony Brook, NY 11794-3651}
\email{guenancia@math.sunysb.edu}
\urladdr{www.math.sunysb.edu/{~}guenancia}


\begin{abstract}
In this note, we prove that on a compact Kähler manifold $X$ carrying a smooth divisor $D$ such that $K_X+D$ is ample, the Kähler-Einstein cusp metric is the limit (in a strong sense) of the Kähler-Einstein conic metrics when the cone angle goes to $0$. We further investigate the boundary behavior of those and prove that the rescaled metrics converge to a cylindrical metric on $\mathbb C^*\times \mathbb C^{n-1}$.
\end{abstract}

\maketitle


\mainmatter
\tableofcontents
\section*{Introduction}

Let $X$ be a compact Kähler manifold of dimension $n$, and $D$ a smooth hypersurface such that $K_X+D$ is ample. A well-known result of Kobayashi \cite{KobR} and Tian-Yau \cite{Tia} asserts the existence of a unique Kähler metric $\om_0$ on $X\smallsetminus D$ with cusp singularities along $D$ and such that $\Ric \om_0=-\om_0$. Recall that $\om_0$ is said to have cusp singularities (or Poincaré singularities) along $D$ if, whenever $D$ is locally given by $(z_1=0)$, $\om_0$ is quasi-isometric to the cusp metric:
\[\om_{\rm cusp} = \frac{i dz_1 \wedge d \bar z_1}{|z_1|^2 \log^2 |z_1|^2} + \sum_{k=1}^n i dz_k \wedge d \bar z_k\]
Moreover, as ampleness is an open condition, there exists $\beta_0>0$ such that for all $0<\beta< \beta_0$, the class $K_X+(1-\beta) D$ is ample. Therefore, results of \cite{CGP, GP, JMR} provide us with a unique Kähler metric $\omb$ on $X\smallsetminus D$ having cone singularities with cone angle $2\pi \beta$ along $D$ and such that $\Ric \omb = -\omb$. Here again, recall that $\omb$ is said to have cone singularities with cone angle $2\pi \beta$ along $D$ if, whenever $D$ is locally given by $(z_1=0)$, $\omb$ is quasi-isometric to the cone metric:
\[\om_{\rm cone} = \frac{i dz_1 \wedge d \bar z_1}{|z_1|^{2(1-\beta)} } + \sum_{k=1}^n i dz_k \wedge d \bar z_k\]
So we have a family of metrics $(\omb)_{0\le \beta<\beta_0}$ on $X\smallsetminus D$ that can actually be viewed as currents on $X$ satisfying the twisted Kähler-Einstein equation:
\[\Ric \omb = - \omb + (1-\beta) [D]\]
A natural question to ask is whether $\om_0$ is the limit, in some suitable sense, of the metrics $\omb$ when $\beta$ goes to $0$. This seems to be a folkore question/result in complex geometry, yet we were not able to find a reference giving a proof of this result.\\

In this note, we show that the answer to the above question is positive, and that the convergence holds both in a weak but global sense and in a strong but local sense:

\begin{thma}
Let $X$ be a compact Kähler manifold carrying a smooth divisor $D$ such that $K_X+D$ is ample.\\ 
Then the Kähler-Einstein metrics $\omb$ with cone angle $2\pi \beta$ along $D$ converge to the Kähler-Einstein cusp metric $\om_0$, both in the weak topology of currents and in the $\mathscr C_{\rm loc}^{\infty}(X\smallsetminus D)$ topology.
\end{thma}

In particular, the metric spaces $(X\smallsetminus D, \omb)$ converge in pointed Gromov-Hausdorff topology to $(X\smallsetminus D, \om_0)$.

The strategy of the proof consists of adapting the stability arguments of \cite{BG} to this setting where the cohomology classes do not evolve in a monotonic manner. 
Once the weak convergence is obtained, it is sufficient to establish a priori estimates for the potentials of $\omb$ in order to get the smooth convergence on the compact subsets of $X\smallsetminus D$. Our main tool will actually be the maximum principle.

Let us note that the exact same proof would actually extend to the case where $D$ is merely a simple normal crossing divisor, yet we chose to stick with the smooth case for the sake of clarity.\\

It turns out that Theorem 1 above does not say much about what happens near the divisor. Typically, it is hard to see why a singularity in $1/|s|^{2(1-\beta)}$ becomes $1/|s|^2 \log^2 |s|^2$ when $\beta \to 0$. To get a better insight of this, one should look at the local case of the punctured disk in $\mathbb C$: there, the metric
$\om_{\beta, \D^*}=\frac{\beta^2 i dz \wedge d\bar z}{|z|^{2(1-\beta)}(1-|z|^{2\beta})^2}$ has conic singularities at $0$ with cone angle $2\pi \beta$ and it has constant negative curvature. Then, when $\beta$ goes to $0$, $\omega_{\beta, \D^*}$ converges pointwise to the Poincaré metric $\omega_{P, \D^*}=\frac{i dz \wedge d\bar z}{|z|^2 \log^2 |z|^2}$.\\

In Section \ref{sec:curv}, we extend this observation to the global case by constructing in each Kähler cohomology class a metric with conic singularities that is uniformly (in $\beta$) equivalent to the (higher dimensional) local model $\om_{\beta, \D^*}$ near the divisor $D$. 
We will also show the the holomorphic bisectional curvature of this model is bounded and that this bound does not depend on $\beta$ as long as $\beta$ is sufficiently close to $0$ (more precisely, we need $\beta \le 1/2$, cf Theorem \ref{curv}).

Then, in Section \ref{sec:unif}, we prove optimal $L^{\infty}$ and Laplacian estimates that rely on a slightly subtle application of the maximum principle as well as on the curvature bound previously established. They show that the conic Kähler-Einstein metric is uniformly equivalent to the model metric:

\begin{thmb}
There exists a constant $C>0$ independent of $\beta$ such that on any coordinate chart $U$ where $D$ is given by $(z_1=0)$, the conic Kähler-Einstein metric $\omb$ satisfies:
\[C^{-1}\,\om_{\beta, \rm mod}  \le \omb \le C \, \om_{\beta, \rm mod}\]
where
\[\om_{\beta, \rm mod}:= \frac{\beta^2 i dz_1 \wedge d\bar z_1}{|z_1|^{2(1-\beta)}(1-|z_1|^{2\beta})^2}+\sum_{k=2}^n i dz_k \wedge d\bar z_k \]
\end{thmb}

The proof of Theorem B is independent of Theorem A; one cannot recover from it the global weak convergence of $\omb$ to $\om_0$, but we could use it to prove the smooth convergence on the compact subsets of $X\smallsetminus D$ from the weak convergence. \\

In \cite[Definition 5.5]{BBGZ}, the authors introduce the notion of convergence in energy (or with respect to the strong topology), for closed positive $(1,1)$-currents with finite energy. It is well-known that the Monge-Ampère operator is discontinuous with respect to the weak topology, and this topology is particularly useful to circumvent that issue. More precisely, the strong topology is the coarsest topology that makes the energy functional continuous. An important observation is that the notion of convergence in energy is defined at the level of the potentials so it only extends to currents living in the same cohomology class, which is precisely not the case of the currents $\omb$. However, Theorem B guarantees that the potentials $\vpb$ of $\omb$ are $\om$-psh for $\beta$ small enough, where $\om \in c_1(K_X+D)$ is a reference Kähler metric, cf \S \ref{sec:energy}. So we can make sense of convergence in energy for $\omb$ (or equivalently $\vpb$), and better we can prove that it actually happens:

\begin{cor}
The currents $\omb$ converge in energy toward $\om_0$.
\end{cor} 

This result is relatively easily deduced from Theorems A and B; it is another instance of the \textit{strong} global convergence of the currents $\omb$ to $\om_0$.\\

In the last part of this paper, we focus on the (rough) asymptotic behavior of $\omb$ near $D$, when $\beta$ tends to $0$. More precisely, we fix a point $p\in D$, and we look at the small neighborhood (included in a coordinate chart) $U_{\beta}:=(0<|z_1|^2<e^{-1/\beta}) \cap (|z_i|^2< 1)$. Up to constants, $U_{\beta}$ corresponds to the (punctured) ball $B_p(\omb,1)$ of radius $1$ centered at $p$ for the metric $\omb$ (or better, its completion). The proper renormalization factor for this metric is $\beta^{-2}$ in this context, and it leads to the convergence (up to taking subsequences) toward a Ricci-flat cylindrical metric, i.e. a metric on $\CC^*\times \CC^{n-1}$ which pulls-back to a constant metric under the universal cover, cf \S \ref{cyl}. The precise result is the following:

\begin{thmc}
Let $(\beta_n)_{n\in \N}$ be a sequence of positive numbers converging to $0$. Then up to extracting a subsequence, there exists a cylindrical metric $\om_{\rm cyl}$ on $\CC^*\times \CC^{n-1}$ such that the metric spaces $(U_{\beta_n}, \beta_n^{-2}\om_{\beta_n})$ converge in pointed Gromov-Hausdorff topology to $(\CC^*\times \CC^{n-1},  \om_{\rm cyl})$ when $n$ tends to $+\infty$.
\end{thmc}

We obtain this result by showing a stronger statement about smooth convergence on compact sets in $\CC^*\times \CC^{n-1}$ under a suitable embedding. The limit of our method, which is based on a priori estimates, is that it only provides relative compactness (and not a limit).
So far, we do not know whether the full family $(\beta^{-2}\omb)$ converges when $\beta\to 0$, as different subsequences could converge to different cylindrical metrics (although two cylindrical metrics have same Riemannian universal cover, they are in general not holomorphically isometric, cf \S \ref{cyl}). 
We suspect that this interesting question is difficult. \\

\noindent
\textbf{Acknowledgements.} I am grateful to Vincent Guedj who suggested this problem to me. Also, I would like to thank Song Sun for very insightful discussions about this paper.

\section{Weak convergence}
\subsection{A first observation}

First, we have to ensure that the family of (closed positive) currents $(\omb)_{0<\beta<\beta_0}$ is relatively compact for the weak topology. Before proving it, let us set up the notations. 

As $K_X+D$ is ample, there exists a Kähler metric $\om \in c_1(K_X+D)$. We pick a section $s$ of $\Ox(D)$ cutting out this hypersurface, then we fix a smooth hermitian metric $|\cdotp|$ on $\Ox(D)$ and we let $\theta$ be its curvature form. 

\begin{rema}
As no assumption is made on $D$, one cannot assume that $|\cdotp|$ can be chosen in such a way that $\theta$ is semipositive. Indeed, let 
 $C$ be a genus $g\ge 2$ curve, $X:=C\times C$ and let $\Delta$ be the diagonal of $X$. By adjunction, $(K_X+\Delta \cdotp \Delta) = (K_{\Delta})=2g-2$ while $(K_X \cdotp \Delta) = (p_1^*K_C + p_2^*K_C \cdotp \Delta) = 2(2g-2)$ so that $(\Delta^2)=2-2g<0$ hence $\Delta$ is not nef (so in particular its cohomology class does not contain any smooth non-negative form). However, an application of Nakai-Moishezon ampleness criterion for surfaces shows that $K_X+\Delta$ is ample, so that $(X, \Delta)$ provides the example we were looking for. \\
\end{rema}

Back to our general setting, we first observe that up to choosing a smaller $\beta_0$, one can always assume that $\om-\beta \theta$ is a Kähler metric. Next, we introduce $h$ a (twisted) Ricci potential of $\om$, i.e. a smooth function satisfying $\Ric \om = -\om + \theta + \ddc h$. We also introduce, for $\beta \in [0, \beta_0)$, the normalized potential $\vpb$ of $\omb$, ie 
\[\omb=\om - \beta \theta + \ddc \vpb\]
and $\sup \vpb =0$. This normalization makes $(\vpb)$ into a precompact family for free; the counterpart is that we lose control on the normalization constant in the Monge-Ampère equation satisfied by $\vpb$, which reads
\begin{equation}
\label{eq:ma}
(\om-\beta \theta + \ddc \vpb)^n = \frac{e^{\vpb+h+\Cb} \om^n}{|s|^{2(1-\beta)}}
\end{equation}
for some $\Cb\in \R$. Actually, it is easy to get an upper bound on $\Cb$:

\begin{lemm}
\label{lem:dom}
There exists $C>0$ independent of $\beta$ such that $C_{\beta}\le C$. 
\end{lemm}

\begin{proof}
Let $V_{\beta}:=\int_X (\om- \beta \theta)^n = c_1(K_X+(1-\beta)D)^n$. Integrating equation \eqref{eq:ma} and applying Jensen's inequality, we get 
\begin{eqnarray*}
\frac{\Vb}{V_0} &=& e^{\Cb} \int_X \frac{e^{\vpb+h} }{|s|^{2(1-\beta)}} \, \om^n/V_0\\
& \ge & e^{\Cb} \cdotp \exp\left(\int_X \left(\vpb+h-(1-\beta) \log |s|^2\right) \, \om^n /V_0 \right)
\end{eqnarray*}
and therefore, there exists $C$ independent of $\beta $ such that $C_{\beta}\le C+\int_X (-\vpb) \om^n/V_0$.
Now, $\vpb$ is $A\om$-psh for $A$ big enough independent of $\beta$, and $\max \vpb =0$, so by basic compactness properties of quasi-psh function, the $L^1$ norm of $\vpb$ is under control, which enables us to conclude the proof of the lemma.
\end{proof}

\subsection{The variational argument}

We know that the family of potentials $(\vpb)$ is precompact, so we can extract weak limits when $\beta \to 0$, and we want to see that all possible limits are the same, equal to $\vp_0$. To prove it, we will use variational arguments inspired by \cite{BG}. Let us recall the setup. 
If $\vp \in \PSH(X,\om-\beta \theta)$ we set 
\[\Gb(\vp)= \Eb(\vp)+\Lb(\vp) \]
where 
\[\Eb(\vp) = \frac{1}{(n+1)\Vb} \sum_{k=0}^n \int_X \vp \,  (\om-\beta \theta)^k \wedge (\om-\beta \theta+\ddc \vp)^{n-k}\]
is the pluricomplex energy attached to the Kähler metric $\om-\beta \theta$ (and whose derivative is the Monge-Ampère operator with respect to this metric) and 
\[\Lb(\vp) = -\log \int_X e^{\vp} \cdot \frac{e^{h+\Cb} \om^n}{|s|^{2(1-\beta)}}\]
Then we know from \cite[Theorem 3.2]{BG} that $\vpb$ is the unique (normalized) maximizer of $\Gb$, for every $\beta \in [0, \beta_0)$. The following lemma expresses the (semi-) continuity properties needed to conclude that $\vpb \to \vp_0$.

\begin{lemm}
\label{lem:sc}
If $\pb \in \PSH(X,\om-\beta \theta)$ is a family converging to $\psi \in PSH(X,\om)$ in $L^1$ topology, then \[\varlimsup_{\beta \to 0} \Gb(\pb) \le \mathcal G_0(\psi)\]
Moreover, any $\vp \in \cap_{0\le \beta <\beta_0} \PSH(X,\om-\beta \theta)\cap \{\mathcal G_0 >- \infty\}$ satisfies: 
\[\lim_{\beta \to 0} \Gb(\vp) = \mathcal G_0(\vp)\]
\end{lemm}

\begin{proof}
Let us begin with the first statement. By Fatou's lemma, we get $\varlimsup \Lb(\pb) \le \mathcal L_0(\psi)$. Moreover, \cite[Lemma 3.6]{BG} gives us precisely the corresponding inequality for the energy: $\varlimsup \Eb(\pb) \le \mathcal E_0(\psi)$. Therefore, an application of the standard inequality $\varlimsup (f+g) \le \varlimsup f + \varlimsup g$ proves our claim. \\

Let us get to the second part. Of course, there is no restriction in assuming that $\vp$ is sup-normalized as the $\mathcal G$ functionals are translation invariant. Now, thanks to Lemma \ref{lem:dom}, we have the following inequality
\[\frac{e^{\vp+h+\Cb}}{|s|^{2(1-\beta)}} \le  \frac{Ce^{\vp+h+C_0}}{|s|^{2}} \]
for some constant $C$. As $\mathcal L_0(\vp)>-\infty$, Lebesgue's domination theorem shows that $\Lb(\vp) \to \mathcal L_0(\vp)$. The energy term can be dealt with in the following way: we choose $C>0$ such that $\om-\beta \theta \le (1+C) \om$. Then, we have, for each $k\in [0, n]$:
\begin{eqnarray*}
(\om-\beta \theta + \ddc \vp)^{k} \wedge (\om-\beta \theta)^{n-k}& \le & (1+C)^{n-k} \left (C \om+(\om+\ddc \vp)\right)^k \wedge \om^{n-k}\\
& \le & (1+C)^{n-k} \sum_{j=0}^k C^{k-j}C^{j}_k(\om + \ddc \vp)^{j} \wedge \om ^{n-j} 
\end{eqnarray*}
and therefore, as $\vp \le 0$, we obtain for all $k\in [0,n]$:
\[0 \le (-\vp) (\om-\beta \theta + \ddc \vp)^{k} \wedge (\om-\beta \theta)^{n-k} \le C \, \sum_{j=0}^n (-\vp) (\om + \ddc \vp)^{j} \wedge \om ^{n-j}  \]
for some $C>0$ independent of $\beta$. Now, as $\mathcal G_0(\vp) $ is finite, then so is $\mathcal E_0(\vp)$, and therefore Lebesgue's domination theorem guarantees that $\Eb(\vp)$ converges to $\mathcal E_0(\vp)$.
\end{proof}

We are now able to prove the first part of the Main Theorem, i.e. the weak convergence of $\omb$ to $\om_0$. As we observed at the beginning of this section, it is sufficient to see that every cluster value of $\vpb$ equals $\vp_0$. So let us consider $\psi$ such a cluster value. We need to see that $\psi$ maximizes $\mathcal G_0$ on $\PSH(X,\om)$.
By the first part of Lemma \ref{lem:sc}, we find:
\begin{equation}
\label{eq1}
\mathcal G_0(\psi) \ge  \varlimsup_{\beta \to 0} \Gb(\vpb) \ge \varlimsup_{\beta \to 0} \Gb(\vp_0)
\end{equation}
the second inequality being derived from the maximizing property of $\vpb$.

The crucial observation now (that would fail in the singular setting for instance) is that the (normalized) maximizer $\vp_0$ of $\mathcal G_0$ is not only $\om$-psh but also $(1-\delta)\om$-psh for some sufficiently small $\delta$. Indeed, we know that $\om_0$ is a current which is a Kähler metric outside $D$ and has cusp singularities along $D$, hence it is a Kähler current, i.e. there exists $\delta>0$ such that $\om+\ddc \vp_0 \ge \delta \om$. In particular, up to choosing a smaller $\beta_0$, the potential $\vp_0$ belongs to the intersection  $\cap_{0\le \beta <\beta_0} \PSH(X,\om-\beta \theta)\cap \{\mathcal G_0 >- \infty\}$. So we can apply the second part of Lemma \ref{lem:sc} to $\vp_0$, and get
\[\varlimsup_{\beta \to 0} \Gb(\vp_0) = \mathcal G_0(\vp_0)\]
Combined with \eqref{eq1}, we find
\[\mathcal G_0(\psi) \ge \mathcal G_0(\vp_0)\]
therefore $\psi$ maximizes $\mathcal G_0$, so it equals $\vp_0$ modulo up to an additive constant. As these two functions are identically normalized, they are equal, and our result is proved.

\begin{rema}
We could have used an alternative simpler (yet less general) argument to show the convergence of $(\vpb)$ to some weak KE metric, based on idea of Tsuji \cite{Tsuji10} expanded further by Song-Tian \cite{SongTian} in particular in \S 4.3. Indeed, an application of the maximum principle (or comparison principle) shows that $\vpb+\beta \log |s|^2$ is increasing when $\beta$ decreases to $0$, and also bounded above. Therefore $(\sup_{\beta\downarrow 0} \vpb)^*$ provides a candidate for a weak KE metric (we wouldn't know if it were $\vp_0$ because of the lack of information near $D$). 

\end{rema}

\section{Smooth convergence on $X\smallsetminus D$}
From now on, we know that $\vpb$ converges to $\vp_0$ for the $L^1$ topology. So all we are left to prove is that the family $(\vpb)$ is precompact for the $\mathscr C_{\rm loc}^{\infty}(X\smallsetminus D)$ topology. Using Ascoli theorem, this amounts to establishing $\mathscr C^k_{\rm loc}(X\smallsetminus D)$ estimates for all $k$, but thanks to Evans-Krylov theory and Schauder interior estimates (the so-called bootstrapping for elliptic PDE's) it is enough to have local $L^{\infty}$ and Laplacian estimates on $X\smallsetminus D$. 

\subsection{The $L^{\infty}_{\rm loc}$ estimate}

\begin{lemm}
\label{est0}
There exists $C>0$ independent of $\beta$ such that
\[\vpb \ge - \log \log^2 |s|^2 - C\]
\end{lemm}
\noindent
Remember that $\sup \vpb =0$, so that the inequality above yields the expected $L^{\infty}_{\rm loc}$ estimate.

\begin{proof}
Let us start with the Monge-Ampère equation \eqref{eq:ma} satisfied by $\vpb$:
\[(\om-\beta \theta + \ddc \vpb)^n = \frac{e^{\vpb+h+\Cb} \om^n}{|s|^{2(1-\beta)}}\]
We will rewrite the equation in terms of the cusp/Poincaré metric $\om_P:=\om - \log \log^2 |s|^2$. Setting $\pb:=\vpb+\log\log^2 |s|^2$ and $\Fb=h+\Cb+\beta \log |s|^2+\log \left( \frac{\om^n}{|s|^2 \log^2 |s|^2 \om_P^n}\right) $, the equation above becomes:
\begin{equation}
\label{eq:ma2}
(\om_P-\beta \theta + \ddc \pb)^n = e^{\pb+\Fb}\om_P^n
\end{equation}

Now, the function $\pb$ is smooth on $X\smallsetminus D$, bounded from below and goes to $+\infty$ near $D$. Therefore it achieves its minimum on $X\smallsetminus D$. At this point, the Hessian of $\pb$ is non-negative. Therefore, we have
\[\inf_{X\smallsetminus D} \pb \ge - \sup_{X\smallsetminus D} \Fb + \inf_{X\smallsetminus D} \log \left(\frac{(\om_P-\beta \theta)^n}{\om_P^n}\right)\]
By Lemma \ref{lem:dom}, $\sup \Fb$ is controlled independently of $\beta$, as is the infimum of the second term as long as $\beta$ is small enough. Therefore $\pb \ge -C$ on $X\smallsetminus D$ for some uniform $C$, from which we deduce the expected inequality.
\end{proof}

\subsection{The Laplacian estimate}
\begin{prop}
\label{lap}
There exist constants $A,C>0$ independent of $\beta$ such that
\[\omb \le   \frac{ C(-\log |s|^2)^A }{|s|^2} \, \om \]
\end{prop}

\begin{proof}
In order to prove this estimate, we will use Siu-Yau's inequality, cf \cite[Lemma 2.2]{CGP} for example:
\begin{lemm}
\label{klem}
Let  $\om,\om'$ be two Kähler forms on a complex manifold $X$, and let $f$ be defined by $\om'^{\,n}=e^f \om^n$. We assume that the holomorphic bisectional curvature of $\om$ is bounded below by some constant $B>0$. Then we have:
\[\Delta' \log \tr_{\om}(\om')\ge \, \frac{\Delta f}{\tr_{\om}(\om')}-B \,  \tr_{\om'}(\om)\]
where $\Delta$ (resp. $\Delta'$) is the Laplace operator attached to $\om$ (resp. $\om'$).
\end{lemm}
We are going to apply this lemma to $\om:=\om_P$, and $\om':=\omb=\om_P-\beta \theta+\ddc \pb$. As $\om_P$ has bounded geometry, its holomorphic bisectional curvature is bounded by some constant $B>0$ on $X\smallsetminus D$, hence we get from equation \eqref{eq:ma2} the following inequality:
\[\Delta_{\omb} \log \tr_{\om_P}(\omb)\ge \, \frac{\Delta_{\om_P} (\pb+\Fb)}{\tr_{\om_P}(\omb)}-B \,  \tr_{\omb}(\om_P)\]

The laplacian of $\Fb$ is bounded on $X\smallsetminus D$, and this bound is uniform in $\beta$. Indeed, $\ddc \Fb = \ddc h - \beta\theta + \ddc \log \left( \frac{\om^n}{|s|^2 \log^2 |s|^2 \om_P^n}\right) $. As $\om_P \ge C^{-1}\om$ for some $C>0$, the first two terms are easily dominated (in absolute value) by a multiple of $\om_P$. Now, the term with the logarithm is smooth in the quasi-coordinates (cf e.g. \cite{KobR}), so in particular its hessian is dominated by a multiple of $\om_P$. As a consequence, $|\Delta_{\om_P} \Fb | \le C$.

Moreover, $\om_P-\beta \theta +\ddc \pb \ge 0$, so $\Delta_{\om_P} \pb \ge \beta \tr_{\om_P} \theta - n \ge -C$ for some uniform $C$. Combining this two estimates with the basic inequality $\tr_{\om_P}(\omb) \cdotp \tr_{\omb}(\om_P)\ge n$, we obtain:
\[\Delta_{\omb} \log \tr_{\om_P}(\omb)\ge -C \,  \tr_{\omb}(\om_P)\]
for some uniform $C$. Furthermore, $\Delta_{\omb} \pb = n+ \beta \tr_{\omb} \theta - \tr_{\omb} \om_P$, which leads to:
\begin{equation}
\label{eq2}
\Delta_{\omb} \left( \log \tr_{\om_P}(\omb)-(C+1) \pb\right) \ge \tr_{\omb}(\om_P) -(C+1) \beta \tr_{\omb}\theta - n(C+1)
\end{equation}
At that point, we need to control the term $\tr_{\omb}\theta$; this would be easy if we could show that $\omb$ dominates some fixed Kähler form (independent of $\beta$), but it turns out that this fact does not seem obvious to prove (essentially because there is no uniform bound on $||\vpb||_{\infty}$). Instead, we can take advantage of the robustness of the method and dominate $ \theta$ by some multiple of $\om_P$ so that (up to choosing a smaller $\beta_0$), we have $(C+1)\beta \tr_{\omb} \theta \le \frac{1}{2} \tr_{\omb} \om_P$ whenever $\beta< \beta_0$. Plugging this inequality into \eqref{eq2}, we get a new constant $C'$ satisfying:
\begin{equation}
\label{eq3}
\Delta_{\omb} \left( \log \tr_{\om_P}(\omb)-(C+1) \pb\right) \ge \frac{1}{2} \tr_{\omb}(\om_P) -C'
\end{equation}
We are now in position to apply the maximum principle. Indeed, as $\omb$ has cone singularities, then $\tr_{\om_P}(\omb)$ is (qualitatively) bounded from above, whereas $-\pb = -\vpb - \log \log^2 |s|^2$ goes to $-\infty$ near $D$ (remember that the potential $\vpb$ of the cone metric is bounded). Therefore the smooth function $H:=\log \tr_{\om_P}(\omb)-(C+1) \pb$ attains its maximum on $X\smallsetminus D$, at a point say $x_0$ (depending on $\beta$).
At that point, inequality \eqref{eq3} combined with the maximum principle yield $ \tr_{\omb}(\om_P)(x_0)\le 2C'$. As a result, we have: 
\begin{eqnarray*}
\log \tr_{\om_P} (\omb) &=& H+(C+1) \pb\\
& \le & \log \tr_{\om_P} (\omb)(x_0)+(C+1) (\pb-\pb(x_0))
\end{eqnarray*}
To control the term involving the logarithm, we use the following inequality
\[\tr_{\om_P}(\omb) \le (\tr_{\omb}(\om_P))^{n-1} e^{\pb+\Fb}\]
which gives, when applied at $x_0$:
\[\log \tr_{\om_P} (\omb)(x_0) \le 2(n-1) \log (2C')+\pb(x_0) + \sup \Fb\]
where we know that $\sup \Fb$ can be controlled uniformly in $\beta$ (cf Lemma \ref{lem:dom}). Combining the two previous inequalities, we obtain
\[\log \tr_{\om_P} (\omb) \le 2(n-1) \log (2C')-C\pb(x_0) + \sup \Fb+(C+1) \pb\]
Remembering that $\pb$ is uniformly bounded from below thanks to Lemma \ref{est0}, we end up with positive constants $A,C$ such that 
\[ \tr_{\om_P} (\omb) \le C (-\log |s|^2)^A\]
from which Proposition \ref{lap} follows.
\end{proof}

\section{The curvature bound}
\label{sec:curv}

In this section, we introduce a particular conic metric which will turn out to behave exactly like the Kähler-Einstein metric (i.e. in a uniform way with respect to the cone angle going to zero). The key property that we will use to establish this fact is the uniform boundedness of its curvature, cf Theorem \ref{curv}.
\subsection{The reference metric}

Let $X$ be a compact Kähler manifold, $\om$ a background Kähler form, $D$ a smooth divisor cut out by an holomorphic section $s$ of the associated line bundle, and let finally $h=|\cdotp|$ be a smooth hermitian metric on $\O_X(D)$ normalized such that $|s|^2 < e^{-1}$. For any $\beta \in (0,1)$, we introduce the following reference metric:
\[\omb:=\om-\ddc \log\left[\frac{1-|s|^{2\beta}}{\beta}\right]^2 \]
So far, $\omb$ is just a closed $(1,1)$ current, but direct computations show the following:

\begin{lemm}
\label{lem:met}
Up to rescaling $h$, $\omb$ is a Kähler form on $X\smallsetminus D$ having conic singularities along $D$ with cone angle $2\pi\beta$ and such that $\omb \ge \frac 1 2 {\om}$. More precisely, we have
\[\omb=\om+\frac{\beta^2}{|s|^{2(1-\beta)}(1-|s|^{2\beta})^2} \, \la D's,D's \ra - \frac{\beta |s|^{2\beta}}{1-\sb} \, \Theta\]
where $D'$ is the $(1,0)$ part of the Chern connection of $(\Ox(D),h)$ and $\Theta$ is its Chern curvature. 
\end{lemm} 

\begin{proof}
The formula is derived from the identity $\ddc \chi \circ \vp= \chi'(\vp) \ddc \vp + \chi''(\vp) d\vp \wedge d^c \vp$ applied to $\vp=|s|^2$ and $\chi(t)= -\log(1-t^{\beta})$. To see $\omb$ defines a (uniform) Kähler form outside $D$ we have to check that $\beta |s|^{2\beta}/(1-\sb)$ can be made arbitrarily small (uniformly with $\beta$) by rescaling $h$, which does not affect the curvature form $\Theta$. And this is a consequence of the fact that the function $f_{\beta}:t\mapsto \frac{\beta t^{\beta}}{1-t^{\beta}}$ is increasing on $(0,1)$, and that $f_{\beta}(t) \to (-\log t)^{-1}$ when $\beta \to 0$. As $(-\log t)^{-1}$ converges to $0$ when $t\to 0$, it guarantees that for any $\delta>0$, one can choose $t_{\delta} \in (0,1)$ such that $f_{\beta}(t) \le \delta$ on $(0, t_{\delta})$ for all $\beta \in (0,1)$. Then, we take $\delta =(2 \sup_X \tr_{\om} \Theta)^{-1}$ and we scale $h$ such that $|s|^2 \le \delta$; by the discussion above, we will have $-\beta \sb (1-\sb)^{-1} \Theta \ge -\frac 1 2 \om$, hence $\omb \ge \frac 12 \om$.
\end{proof}

A simple but fundamental remark lies in the fact the function $(1-\sb)/\beta$ converges in $L^1$ topology to $-\log |s|^2$ on $X$ when $\beta \to 0$. In particular, $\omb$ converges weakly to the Poincaré type metric $\om_0$ on $X\smallsetminus D$ given by $\om_0:=\om-\ddc \log \log^2|s|^2$. Moreover, it can be checked easily (using Arzelà-Ascoli theorem combined with $(i)$ in Lemma \ref{lem:elem} below) that this convergence actually happens in $\mathscr C_{\rm loc}^{\infty}(X\smallsetminus D)$.\\

The main result of this section is the following:

\begin{theo}
\label{curv}
There exists a constant $C>0$ depending only on $X$ such that for all $\beta \in (0,\frac 12]$, the holomorphic bisectional curvature of $\omb$ is bounded by $C$.
\end{theo}

The rest of this section is devoted to the proof of this statement. We should point out that the proof gets simplified a lot when one assumes that $\beta= 1-1/m$ for an integer $m$ (that will eventually go to $+\infty$), as in that case one can pull back the metric to an orbifold cover where it has uniformly bounded geometry (with respect to $m$). However, in the general conic case, one cannot use such an argument anymore.\\

Let us begin by setting up some notations, and operate a few simplifications for the computations to follow. 
We fix a point $p\in X\smallsetminus D$, and it is a very standard fact that we can find some local holomorphic coordinates around $p$ say $(z_1, \ldots, z_n)$ such that the metric $h=e^{-\vp}$ on $L$ satisfies $\vp(p)=0$ and $d\vp(p)=0$. We can also ass
ume that these coordinates trivialize $L$, and that $s=z_1$ there. We denote by $(g_{i\bar j})$ the Riemannian metric induced by $\omb$ in these coordinates, and we are interested in its curvature tensor
\[\rij = -\frac{\d^2 g_{i\bar j}}{\d z_k \d \bar z_l} +\sum_{\alpha, \beta} g^{\alpha \bar \beta}\frac{\d  g_{i \bar \alpha}}{\d z_k}  \frac{\d g_{\beta \bar j}}{\d \bar z_l} \]
So we consider two tangent vectors $u=\sum u_i \frac{\d}{\d z_i}$ and $v= \sum v_k \frac{\d}{\d z_k}$ whose norm computed with respect to $\omb$ is equal to $1$: $|u|^2_{\omb}=|v|^2_{\omb}=1$. Our ultimate goal is to prove a bound
\[\left|\rij u_i \bar u_j v_k \bar v_l \right| \le C\]

\subsection{A precise expression of the metric}
We will now express our metric $(g_{i\bar j})$ in the coordinates introduced above. We know that \[\la D's, D's \ra=e^{-\vp}\left[ (dz_1+z_1 \frac{\d \vp}{\d z_k}dz_k) \wedge (d\bar z_1+\bar z_1 \frac{\d \vp}{\d \bar z_l} d \bar z_l)\right] \] Therefore, there exist smooth functions $a,b_k,c_{k\bar l}$ $(2 \le k,l \le n)$ vanishing at $p$ up to order 2 (i.e. they vanish at $p$, and so do their differential) such that 
\[\la D's, D's \ra = (1+a) dz_1 \wedge d \bar z_1 + \sum_{k>1} (z_1b_k \, dz_k \wedge d \bar z_1+\overline{ z_1  b_k}\, dz_1 \wedge d \bar z_k)+ |z_1|^2 \sum_{k,l>1} c_{k\bar l}dz_k \wedge d \bar z_l\]

Before going any further, let us introduce some convenient notations. First, in all the following, we set $t:=|s|^2$, and we define the following two functions:
\[A(t):= \beta^2 t^{\beta-1}(1-t^{\beta})^{-2} \quad \textrm{and} \quad B(t):=\beta t^{\beta}(1-t^{\beta})^{-1}\]

\vspace{3mm}
\noindent
With these expressions at hand, we can get an concise expression of the coefficients of our metric (here, $k,l$ vary between $2$ and $n$)
\begin{eqnarray}
\label{eq:g}
g_{1\bar 1}&=& (1+a) A(t)-B(t) \Theta_{1\bar 1}+\gt_{1\bar 1} \nonumber \\
g_{k\bar 1} &=& z_1 b_{k} A(t)-B(t) \Theta_{k\bar 1}+\gt_{k\bar 1} \\
g_{k \bar l} &=& |z_1|^2 c_{k\bar l} A(t)-B(t) \Theta_{k\bar l}+\gt_{k\bar l} \nonumber
\end{eqnarray}
where $\gt$ is the Riemannian metric associated to the background Kähler form $\om$, and $\Theta_{i \bar j}$ are the components of the form $\Theta$ in the considered coordinates. Given the expression of $g$ above, one can deduce the following estimates for the inverse metric of $g$, valid at $p$:
\begin{eqnarray}
\label{eq:ginv}
g^{1\bar 1}&=& A(t)^{-1}(1+O(A(t)^{-1})) \nonumber \\
g^{k\bar 1} &=& O(A(t)^{-1})  \\
g^{k \bar l} &=& O(1) \nonumber
\end{eqnarray}
Indeed, $(g_{i\bar j}(p))=A(t)E_{1\bar 1}+O(1)$ and $A(t)$ tends to $+\infty$ when $t\to 0$ (in a non uniform way with respect to $\beta$ though). If $d(t)$ is the determinant of $(g_{i\bar j})_{i,j\ge 2}$ (it also depends on $p$ of course), then we have $\det g =A(t) d(t)+O(1)$, hence by Cramer's formula, $g^{1\bar 1}=d(t)(A(t) d(t)+O(1))^{-1}$. As $d(t)$ is uniformly bounded away from $0$, we obtain the expected result. The second estimate is a consequence of the fact that the $(k,1)$-minor of $(g_{i\bar j})$ is a $O(1)$, combined with the estimate on the determinant above. The last estimate is obtained in the same way.\\

In order to estimate the curvature tensor of $g$, we will certainly need to study the functions $A,B$ and their derivatives. So we have collected a few computations about these functions:

\begin{lemm}
\label{lem:elem}
We have the following:
\begin{enumerate}
\item[$(i)$] Given any $t\le 1/4$, we have:
\[\frac{\beta}{1-t^{\beta}} \le 1\]
\item[$(ii)$] For any $t\in (0, +\infty)$, we have:
\[\frac{1-t^{\beta}}{\beta} \le - \log t\]
\item[$(iii)$] When $t\to 0$, we have:
\[A(t)=O(t^{\beta-1}), \, A'(t)=O(t^{\beta-2}), \, A''(t)=O(t^{\beta-3})\]
and 
\[B(t)=O(1), \,B'(t)=O(t^{\beta-1}), \, B''(t)=O(t^{\beta-2})\]
\item[$(iv)$] For $k=1\ldots n$, we have: 
\[g^{k\bar 1}=O(t^{1-\beta}(-\log t)^{2})\]
\end{enumerate}
All the $O$  above are uniform in $\beta$.\\
\end{lemm}

\begin{proof}
For $(i)$, the function $t\mapsto1-t^{\beta}-\beta$ is decreasing, and vanishes at $t_{\beta}=(1-\beta)^{1/\beta}$. Moreover, $\beta\mapsto t_{\beta}$ is a decreasing function too, and its value at $\beta=1/2$ is $1/4$, hence the result. As for $(ii)$, we consider the function $f:t\mapsto -\beta \log t+t^{\beta}-1$. It is smooth outside $0$ where it is $+\infty$. Moreover, its derivative is $\beta /t(t^{\beta}-1)$ so that the minimum of our function is attained at $t=1$, where the function vanishes.

For $(iii)$, the fact that $B(t)=O(1)$ follows from $(i)$. The other estimates for the derivatives of $B$ are a consequence of those for $A$ as $B'(t)=A(t)$. The estimate for $A(t)$ follows from $(i)$. Moreover, 
\begin{eqnarray*}
A'(t)&=&\beta^2t^{\beta-2}(1-t^{\beta})^{-3}\left[(\beta-1)+(\beta+1)t^{\beta}\right]\\
&=&\left(\beta^3(1-t^{\beta})^{-3}(1+t^{\beta})-\beta^2(1-t^{\beta})^{-2}\right)t^{\beta-2}
\end{eqnarray*}
and the expected result is a consequence of $(i)$. Finally, 
\begin{eqnarray*}
A''(t)&=&\beta^2t^{\beta-3}(1-t^{\beta})^{-4}\left[(\beta-1)(\beta-2)+4(\beta^2-1)t^{\beta}+(\beta^2+3\beta+2) t^{2\beta}\right]\\
&=&\beta^2t^{\beta-3}(1-t^{\beta})^{-4}\left[\beta^2(1+4t+t^{2\beta})+(1-t^{\beta})^2+3\beta(1-t^{\beta})(1+t^{\beta)}\right]\\
&=&\left(\beta^4(1-t^{\beta})^{-4}(1+4t+t^{2\beta})+\beta^2(1-t^{\beta})^{-2}+3\beta^3(1-t^{\beta})^{-3}(1+t^{\beta})\right)t^{\beta-3}
\end{eqnarray*}
and we conclude by $(i)$ once again.

Finally, $(iv)$ is a formal consequence of \eqref{eq:ginv} and $(ii)$.
\end{proof}

\subsection{Curvature estimates}
In the following, the indexes $i,j,k,l$ will implicitly be assumed to be different from $1$. Also, as $d( \chi \circ \vp)=\chi'(\vp) d\vp$ and $\ddc( \chi \circ \vp )= \chi''(\vp) d\vp \wedge d^c \vp +\chi'(\vp) \ddc \vp$, we have at $p$:
\begin{eqnarray}
\label{eq:der}
d(A(t))(p)&=&A'(t)\bar z_1 dz_1 \\
\ddc A(t)(p)&=&(tA''(t)+A'(t)) dz_1\wedge d\bar z_1-tA'(t) \Theta \nonumber
\end{eqnarray}
and similarly for $B$.\\

\noindent
Let us now start by estimating the first derivatives of $g$. By \eqref{eq:g} and \eqref{eq:der}, we have:
\[dg_{1\bar 1}(p)=(A'(t)\bar z_1-B'(t)\bar z_1 \Theta_{1\bar 1})\, dz_1+O(1)\]
Therefore, 
\begin{equation}
\label{eq:der1}
\frac{\d g_{1\bar 1}}{\d z_1}(p)= O(t^{\beta-3/2}) \quad \textrm{and} \quad \frac{\d g_{1\bar 1}}{\d z_j}(p)=O(1)
\end{equation}
Similarly, it follows from \eqref{eq:g} that
\begin{equation}
\label{eq:der2}
\frac{\d g_{1\bar j}}{\d z_1}(p)= O(t^{\beta-1/2}) \quad \textrm{and} \quad \frac{\d g_{1\bar j}}{\d z_k}(p)=O(1)
\end{equation}
and 
\begin{equation}
\label{eq:der3}
\frac{\d g_{j\bar k}}{\d z_1}(p)= O(t^{\beta-1/2}) \quad \textrm{and} \quad \frac{\d g_{j\bar k}}{\d z_l}(p)=O(1)
\end{equation}
As for the second derivatives, we have
\begin{eqnarray*}
\ddc g_{1\bar 1}(p)&=&\left[(tA''(t)+A'(t))+(tB''(t)+B'(t))\Theta_{1\bar 1}\right] dz_1\wedge d \bar z_1 \\
&&-tA'(t)\Theta-B'(t)(\bar z_1 dz_1 \wedge O(1)+O(1)\wedge z_1 d\bar z_1)+O(1)
\end{eqnarray*}
so that 
\begin{equation}
\label{eq:der4}
\frac{\d^2 g_{1\bar 1}}{\d z_1 \d \bar z_1}(p)= (tA''(t)+A'(t))+O(t^{\beta-1}) 
\end{equation}
as well as 
\begin{equation}
\label{eq:der5}
\frac{\d^2 g_{1\bar 1}}{\d z_j \d \bar z_1}(p)= O(t^{\beta-1})  \quad \textrm{and} \quad \frac{\d^2 g_{1\bar 1}}{\d z_j \d \bar z_k}(p)=  O(t^{\beta-1}) 
\end{equation}
Similarly, we get
\begin{equation}
\label{eq:der6}
\frac{\d^2 g_{1\bar j}}{\d z_1 \d \bar z_1}(p)= O(t^{\beta-1})\, , \quad   \frac{\d^2 g_{1\bar j}}{\d z_1 \d \bar z_k}(p)=  O(t^{\beta-1/2}) \quad \textrm{and} \quad \frac{\d^2 g_{1\bar j}}{\d z_k \d \bar z_l}(p)=  O(t^{\beta-1/2})
\end{equation}
and finally
\begin{equation}
\label{eq:der7}
\frac{\d^2 g_{i\bar j}}{\d z_1 \d \bar z_1}(p)= O(t^{\beta-1})\, , \quad   \frac{\d^2 g_{i\bar j}}{\d z_1 \d \bar z_k}(p)=  O(t^{\beta-1/2}) \quad \textrm{and} \quad \frac{\d^2 g_{i\bar j}}{\d z_k \d \bar z_l}(p)=  O(1)
\end{equation}

\vspace{3mm}
We are now ready to estimate the bisectional curvature of $g$. So we take two tangent vectors $u=\sum u_i \frac{\d}{\d z_i}$ and $v= \sum v_k \frac{\d}{\d z_k}$ satisfying $|u|^2_{\omb}=|v|^2_{\omb}=1$. In particular, there exists a uniform constant $C>0$ such that
\begin{equation}
\label{vect}
|u_1|^2 \le C A(t)^{-1} \quad \textrm{and} \quad |u_j|^2 \le C 
\end{equation}
and likewise for $v$. We ultimately want to bound the sum $\sum_{i,j,k,l} \rij u_i \bar u_j v_k \bar v_l$; we are going to proceed term by term splitting the cases according to the number of times where $1$ appears in $(i,j,k,l)$.\\

$\bullet$ \textbf{Case 1:} $\{i,j,k,l\}=\{1\}$.

\noindent
This is the most singular term. We split $R_{1\bar 1 1 \bar 1}$ into to terms:
\[R_{1\bar 1 1 \bar 1}=\left(- \frac{\d^2 g_{1\bar j}}{\d z_1 \d \bar z_1}+g^{1\bar 1} \left|\frac{\d g_{1\bar 1}}{\d z_1} \right|^2\right)+\sum_{(\alpha, \beta)\neq(1,1)} g^{\alpha \bar \beta}\frac{\d  g_{1 \bar \alpha}}{\d z_1}  \frac{\d g_{\beta \bar 1}}{\d \bar z_1} \]
Let us deal first with the second term. By \eqref{eq:ginv} and \eqref{eq:der1}-\eqref{eq:der2}, this term is either a $O(A(t)^{-1}t^{2\beta -2})$ if $1\in \{\alpha, \beta\}$ or a $O(t^{2\beta-1})$ else. Whenever we multiply it by $|u_1|^2 |v_1|^2$, it becomes either a $O((A(t)^{-3}t^{2\beta -2})$ or a $O(A(t)^{-2}t^{2\beta-1})$ thanks to \eqref{vect}. As $A(t)^{-1}=O(t^{1-\beta}(-\log t)^2)$, our term is dominated by $t^{1-\beta}(-\log t)^6$ or $t (-\log t)^4$, so in particular it is uniformly bounded.

\noindent
The first summand is subtler to deal with, as the estimate \eqref{eq:der1} is not precise enough to conclude. So we write:
\[\frac{\d g_{1\bar 1}}{\d z_1}(p)=A'(t)\bar z_1+O(t^{\beta-1/2})\]
so that 
\begin{eqnarray*}
g^{1\bar 1} \left|\frac{\d g_{1\bar 1}}{\d z_1} \right|^2  &=&A(t)^{-1}\left((1+O(A(t)^{-1})\right)\left(tA'(t)^2+O(t^{2\beta-2})\right)    \\
&=&tA(t)^{-1}A(t)^2+O(A(t)^{-1}t^{2\beta-2})+O(A(t)^{-2}tA'(t)^2)\\
&=&tA(t)^{-1}A(t)^2+O(t^{\beta-1}(-\log t)^2)+O(t^{-1}(-\log t)^{4})\\
&=&tA(t)^{-1}A(t)^2+O(t^{-1}(-\log t)^{4})
\end{eqnarray*}
Remembering from \eqref{eq:der4} that:
\[\frac{\d^2 g_{1\bar 1}}{\d z_1 \d \bar z_1}(p)= tA''(t)+A'(t)+O(t^{\beta-1}) \]
we end up with 
\[- \frac{\d^2 g_{1\bar j}}{\d z_1 \d \bar z_1}+g^{1\bar 1} \left|\frac{\d g_{1\bar 1}}{\d z_1} \right|^2=-(tA''(t)+A'(t))+tA(t)^{-1}A'(t)^2+O(t^{-1}(-\log t)^{4})\]
The dominant term looks like it will give rise to unbounded curvature, but actually some cancellations come up (as they should, in view of the fact that the Poincaré metric on the punctured disk has constant curvature). More precisely, an easy though tedious computation based on the expressions of $A,A',A''$ given in the proof of Lemma \ref{lem:elem} shows that \[-(tA''(t)+A'(t))+tA(t)^{-1}A'(t)^2=-2A(t)^2\] Therefore, 
\[|u_1|^2|v_1|^2\left(- \frac{\d^2 g_{1\bar j}}{\d z_1 \d \bar z_1}+g^{1\bar 1} \left|\frac{\d g_{1\bar 1}}{\d z_1} \right|^2\right)=O(1)+O(t^{1-2\beta}(-\log t)^{4})\]
hence $|R_{1\bar 1 1\bar 1}| |u_1|^2|v_1|^2 \le C$ for some uniform $C$ as long as $\beta$ varies in $(0, \beta_0]$ with $\beta_0<1/2$.

\noindent
However, if $\beta$ varies in $[\beta_0, 1/2]$, we use the majoration $A(t)^{-1} \le C_0 t^{1-\beta_0}$ where $C_0$ only depends on $\beta_0$, which is finer than $A(t)^{-1}=O(t^{1-\beta} (-\log t)^2)$, and all the expected bounds follow easily. This observation can be applied in the following cases two, so we will not repeat it each time as we will implicitly assume that $\beta$ varies in $(0, \beta_0]$ for some fixed $\beta_0<1/2$.\\

$\bullet$ \textbf{Case 2:} three indexes are equal to $1$.

\noindent
By the symmetries of the curvature tensor, it is enough to consider $R_{1\bar 1 1 \bar k}$. 

\noindent
First we have from \eqref{eq:der5}:
\[\frac{\d^2 g_{1\bar 1}}{\d z_1 \d \bar z_k}(p)= O(t^{\beta-1})   \]
Now remember from \eqref{eq:der1}-\eqref{eq:der2} that
\[\frac{\d g_{1\bar 1}}{\d z_1}(p)= O(t^{\beta-3/2}) \quad \textrm{and} \quad \frac{\d g_{1\bar 1}}{\d z_k}(p)=O(1)\]
and
\[\frac{\d g_{1\bar \alpha}}{\d z_1}(p)= O(t^{\beta-1/2}) \quad \textrm{and} \quad \frac{\d g_{\beta \bar 1}}{\d \bar z_k}(p)=O(1)\]
so that
\[ g^{\alpha \bar \beta}\frac{\d  g_{1 \bar \alpha}}{\d z_1}  \frac{\d g_{\beta \bar 1}}{\d \bar z_k} = 
\begin{cases}
O(t^{\beta-1/2}) & \textrm{if \,} \alpha, \beta \neq 1\\
O(A(t)^{-1}t^{\beta-3/2})  & \textrm{if \,} \alpha=1, \beta \neq 1\\
O(1)  & \textrm{if \,} \alpha \neq 1, \beta=1 \\
O(A(t)^{-1}t^{\beta-3/2})  & \textrm{if \,} \alpha =\beta=1 
\end{cases} \]
So in any case, this quantity is a $O(A(t)^{-1}t^{\beta-3/2})$. Therefore
\begin{eqnarray*}
R_{1\bar 1 1 \bar k} \,  |u_1|^2 v_1 \bar v_k&=&O(A(t)^{-3/2}t^{\beta-1})+O(A(t)^{-5/2}t^{\beta-3/2})\\
&=&O(t^{1/2-\beta/2}(-\log t)^3)+O(t^{1-3\beta/2}(-\log t)^5) \\
&=&O(1)
\end{eqnarray*}

$\bullet$ \textbf{Case 3:} two indexes are equal to $1$.

\noindent
Again, using the symmetries, one can reduce to estimating the following two quantities: $R_{1 \bar 1 k \bar l}$ and $R_{1 \bar k 1 \bar l}$. 

\noindent
Let us start with the first one. We know from \eqref{eq:der1}, \eqref{eq:der2} that
\[ \frac{\d g_{1\bar 1}}{\d z_k}(p)=O(1)  \quad  \textrm{and} \quad \frac{\d g_{1\bar \alpha}}{\d z_k}(p)=O(1)\]
For the second derivatives, the estimate $ \frac{\d^2 g_{1\bar 1}}{\d z_k \d \bar z_l}(p)=  O(t^{\beta-1}) $ from \eqref{eq:der5} is not sufficient as $A(t)t^{\beta-1}$ is not uniformly bounded above. So we have to be more precise, and extract from \eqref{eq:g} the refined information: 
\[ \frac{\d^2 g_{1\bar 1}}{\d z_k \d \bar z_l}(p)=  \frac{\d^2 a}{\d z_k \d \bar z_l}(p) \cdotp A(t)+O(1)= O(A(t))\]
hence
\[R_{1 \bar 1 k \bar l} \, |u_1|^2 v_k \bar v_l = O(1)\]
thanks to \eqref{vect}.\\

\noindent
The second case is slightly more involved. By \eqref{eq:der6}, we get
\[\frac{\d^2 g_{1\bar k}}{\d z_1 \d \bar z_l}(p)=  O(t^{\beta-1/2}) \] and also
\[ g^{\alpha \bar \beta}\frac{\d  g_{1 \bar \alpha}}{\d z_1}  \frac{\d g_{\beta \bar k}}{\d \bar z_l} = 
\begin{cases}
O(t^{\beta-1/2}) & \textrm{if \,} \alpha, \beta \neq 1\\
O(A(t)^{-1}t^{\beta-3/2})  & \textrm{if \,} \alpha=1, \beta \neq 1\\
O(A(t)^{-1}t^{\beta-1/2})  & \textrm{if \,} \alpha \neq 1, \beta=1 \\
O(A(t)^{-1}t^{\beta-3/2})  & \textrm{if \,} \alpha =\beta=1 
\end{cases} \]
which is in any case dominated by $O(A(t)^{-1}t^{\beta-3/2}) $. As a result, 
\begin{eqnarray*}
R_{1\bar k 1 \bar l} \,  u_1\bar u_k v_1 \bar v_l&=&O(A(t)^{-1}t^{\beta-1/2})+O(A(t)^{-2}t^{\beta-3/2})\\
&=&O(t^{1/2}(-\log t)^2)+O(t^{1/2-\beta}(-\log t)^4) \\
&=&O(1)
\end{eqnarray*}
as $\beta \in (0, \beta_0]$ with $\beta_0<1/2$.\\

$\bullet$ \textbf{Case 4:} one index is equal to $1$.

\noindent
We only have to consider $R_{1 \bar j k \bar l}$. To start with, \eqref{eq:der6} provides us with:
$\frac{\d^2 g_{1\bar j}}{\d z_k \d \bar z_l}(p)=  O(t^{\beta-1/2})$ which is not precise enough as $A(t)^{-1/2}t^{\beta-1/2}$ is not uniformly bounded. So we go back to the precise expression \eqref{eq:g} to get:
\begin{eqnarray*}
\frac{\d^2 g_{1\bar j}}{\d z_k \d \bar z_l}(p)&=&\bar z_1 A(t) \frac{\d ^2 \bar b_j}{\d z_k \d \bar z_l}(p)+O(1)\\ 
&=& O(t^{1/2}A(t))
\end{eqnarray*}
Moreover, we have:
\begin{equation}
\label{eq:last}
 g^{\alpha \bar \beta}\frac{\d  g_{1 \bar \alpha}}{\d z_k}  \frac{\d g_{\beta \bar j}}{\d \bar z_l} = O(1) 
 \end{equation}
independently of whether $\alpha$ or $\beta$ is equal to $1$, thanks to \eqref{eq:der1}-\eqref{eq:der3}. Therefore, 
\begin{eqnarray*}
R_{1\bar j k \bar l} \,  u_1\bar u_j v_k \bar v_l &=& O(t^{1/2}A(t)^{1/2})+O(1)\\
&=&O(t^{\beta/2})+O(1) \\
&=&O(1)\\
\end{eqnarray*}

$\bullet$ \textbf{Case 5:} no index is equal to $1$.

\noindent
In that case, it follows from \eqref{eq:der7} that 
\[\frac{\d^2 g_{i\bar j}}{\d z_k \d \bar z_l}(p)=  O(1)\]
which we combine with the second estimates of \eqref{eq:der2}-\eqref{eq:der3} to obtain
\[R_{i\bar j k \bar l} \,  u_i\bar u_j v_k \bar v_l =O(1)\]
This concludes the proof of Theorem \ref{curv}.

\section{Proof of Theorem B}
\label{sec:unif}
Let us start by recalling the setting of the first sections, with slightly different notations. $X$ is still a compact Kähler manifold and $D$ is a smooth divisor such that $K_X+D$ is ample. Therefore, for $\beta>0$ small enough, there exists a unique metric $\omv\in c_1(K_X+(1-\beta)D)$ such that $\Ric \omv = -\omv + (1-\beta)[D]$. We fix a reference Kähler form $\om \in c_1(K_X+D)$, and denote by $\theta$ a smooth representative of $c_1(D)$, that we choose to be the curvature of our fixed smooth hermitian metric on $\Ox(D)$. We want to compare $\omv$ to the reference conic metric $\omb=\om-\ddc \log\left[\frac{1-|s|^{2\beta}}{\beta}\right]^2$ constructed in \S \ref{sec:curv}. We proceed in two steps; first we compare the potentials (zero order estimate) and then the metrics themselves (Laplacian estimate).\\

The first thing to do is consider the suitably normalized potential of $\omv$. For cohomological reasons, there exists $\tilde \vp_{\beta}$ such that $\omv= \om-\beta \theta+\ddc \tilde \vp_{\beta}$. Now, given the Kähler-Einstein equation satisfied by $\omv$, there exists a volume form $dV$ (independent of $\beta$) and a constant $C_{\beta}$ such that $\omv^n= \frac{e^{\tilde \vp_{\beta}+C_{\beta}}dV}{|s|^{2(1-\beta)}}$. We then normalize $\tilde\vp_{\beta}$ so that $C_{\beta}=0$. As a result, $\omv$ is solution of the following equation 
$$(\om-\beta \theta+\ddc \tilde \vp_{\beta})^n= \frac{e^{\tilde \vp_{\beta}}dV}{|s|^{2(1-\beta)}}$$
If we introduce the potential $\pb:=-\log\left[\frac{1-|s|^{2\beta}}{\beta}\right]^2$ of $\omb$, then we can reformulate the equation above in terms of the potential $\vpb:=\tilde \vp_{\beta}-\pb$:
\begin{equation}
\label{eq:norm}
(\omb-\beta \theta+\ddc \vpb)^n= e^{\vpb+\Fb}\omb^n
\end{equation}
where 
\[\Fb=\pb+\log\left( \frac{dV}{\sbb \omb^n}\right)\]
From Lemma \ref{lem:met}, we see that $\sbb \omb^n = \frac{\beta^2}{(1-\sb)^2}e^{O(1)}dV$, which implies that $\Fb=O(1)$. This observation enables us to prove the following:

\begin{prop}
\label{prop:est0}
There exists $C>0$ such that for all $\beta>0$ small enough, we have $\sup_{X\smallsetminus D} |\vpb| \le C$.
\end{prop}

\begin{proof}
We know that $\vpb$ is qualitatively bounded, and in order to make it quantitative, we would like to use the maximum principle. Unfortunately, $X\smallsetminus D$ is not compact, and $\omb$ is not complete either, so we a little more work is needed. We introduce for each $\ep>0$ the function $\chi_{\beta, \ep}:=\vpb+\ep \log |s|^2$. By construction, it is bounded above at attains its maximum on $X\smallsetminus D$, at a point say $x_{\ep}$. Using that $\ddc \vpb = \ddc \cbe+ \ep \theta$, we obtain that  $\ddc \vpb(x_{\ep})\le \ep \theta(x_{\ep})$. As a consequence, $(\omb-\beta \theta + \ddc \vpb)^n(x_{\ep}) \le (\omb-(\beta+\ep)\theta)^n(x_{\ep})\le 2^n \omb^n(x_{\ep})$ for $\ep, \beta$ small enough. Therefore, $(e^{\vpb+\Fb}\omb^n)(x_{\ep}) \le 2^n \omb^n(x_{\ep})$, or also $\vpb(x_{\ep})\le -F_{\beta}(x_{\ep})+n \log 2$, and thus $\vpb(x_{\ep})\le -\inf \Fb+n \log 2$. Take now an arbitrary $x\in X\smallsetminus D$; using the definition of $x_{\ep}$ and the fact that $|s|^2<1$, we end up with:
\begin{eqnarray*}
\vpb(x)&=&\cbe(x)-\ep (\log |s|^2)(x)\\
& \le & \vpb(x_{\ep})+\ep (\log |s|^2)(x_{\ep})-\ep (\log |s|^2)(x) \\
& \le & -\inf \Fb+n \log 2 -\ep (\log |s|^2)(x)
\end{eqnarray*}
Making $\ep$ go to zero ($x$ is fixed), we finally obtain $\sup \vpb \le C$ for some uniform constant $C$. 

\noindent
For the minimum, we can reproduce the same argument with $\tilde \chi_{\beta, \ep}:=\vpb-\ep \log |s|^2$, and obtain that $\inf \vpb \ge - \sup \Fb+ n \log 2\ge -C$ which proves the proposition.
\end{proof}

With this estimate at hand, one can take advantage of the boundedness of the curvature of $\omb$ to get the Laplacian estimate, which is the content of Theorem B:

\begin{prop}
There exists $C>0$ such that for all $\beta$ small enough, we have
\[C^{-1} \omb \le \omv \le C \omb\]
\end{prop}

\noindent
One may probably emphasize again that we already know the existence of such a constant for each $\beta$, and the new feature is that one can choose $C$ to be independent of $\beta$.

\begin{proof}
The key inequality that we are going to rely on is Chern-Lu's formula applied to the identity function $\mathrm{id}:(X\smallsetminus D, \omv) \to (X\smallsetminus D, \omb)$. By definition, $\Ric \omv = - \omv$, and we know from Proposition \ref{curv} that there exists a universal constant $B>0$ such that $\mathrm{Bisec} \, \omb \le B$, so Chern-Lu formula yields:
\[\Delta_{\omv}(\log \tra) \ge -1 - 2B \tra\]
Now, $\omv= \omb-\beta \theta + \ddc \vpb$ and there exists $M>0$ such that $\theta \le M \omb$. Take $\beta \le 1/2M$; then $-\ddc \vpb \ge \frac 1 2 \omb - \omv$, and therefore $-\Delta_{\omv} \vpb \ge \frac 1 2 \tra - n$. As a result, if we set $A:=4(B+1)$, we get:
\[\Delta_{\omv}(\log \tra-A\vpb) \ge 2 \tra-C\]
where $C=1+4n(B+1)$. We set $H:=\log \tra -A\vpb$; we want to apply the maximum principle to this function, but as for the zero order estimate, we need to be cautious. So for each $\ep>0$, we introduce $\He:=H+\ep \log |s|^2$; by the previous inequality, this function satisfies $\Delta_{\omv}  \He  \ge 2 \tra - C - \ep \, \tr_{\omv}\theta$, and if one assumes that $\ep<1/M$, then $\ep \, \tr_{\omv}\theta \le \tra$ hence
\[\Delta_{\omv}  \He \ge \tra -C\]
As $\He$ tends to $-\infty$ near $D$ and since $H$ is (qualitatively) bounded on $X\smallsetminus D$, we can pick a point $x_{\ep}$ such that $\He$ attains its maximum at $x_{\ep}$. From the maximum principle, we get that $\tra(x_{\ep})\le C$. Therefore, if $x\in X\smallsetminus D$, one has:
\begin{eqnarray*}
\log \tra(x) &=& H(x)-A \vpb(x)\\
& = & \He(x)+A\vpb(x)-\ep (\log |s|^2)(x)\\
& \le & \He(x_{\ep})+A\vpb(x)-\ep (\log |s|^2)(x)\\
& \le & \log \tra(x_{\ep})-A\vpb(x_{\ep})+A\vpb(x)+\ep (\log |s|^2 )(x_{\ep})-\ep (\log |s|^2)(x)\\
& \le & C -\ep (\log |s|^2)(x)
\end{eqnarray*} 
as $\sup |\vpb|$ is under control by Proposition \ref{prop:est0} and $\log |s|^2 \le 0$. Making $\ep$ go to zero, we obtain the desired result.
\end{proof}

\section{Convergence in energy}
\label{sec:energy}

We refer to \cite{BBGZ, BBEGZ} for any further details/applications regarding the notions involved in this section.

Let $X$ be a compact Kähler manifold, $\om$ a Kähler metric, $V=\int_X \om^n$ its volume. Given a \textit{bounded} $\om$-psh function $\psi$, the energy of $\psi$ is 
is defined using Bedford-Taylor product: $$E(\psi) = \frac{1}{(n+1)V}\sum_{k=0}^n \int_X \psi \, (\om+\ddc \psi)^k \wedge \om^{n-k}$$ When $\vp$ is an arbitrary $\om$-psh function, one defines
\[E(\vp):=\inf \{E(\psi) \, | \psi \in \PSH(X,\om) \cap L^{\infty}(X), \,\psi \ge \vp\}\]
and the space of finite energy function is $\mathcal E^1(X, \om):=\{\vp \in PSH(X, \om), E(\vp)>-\infty\}$. As the Monge-Ampère operator is not continuous with respect to the usual $L^1$ topology, it is convenient to introduce the following stronger topology, cf \cite[Definition 2.1]{BBEGZ}:


\begin{defi}
The strong topology on $\mathcal E^1(X,\om)$ is defined as the coarsest refinement of the weak topology such that $E$ becomes continuous.
\end{defi}

With this terminology, we will say that a sequence or family $(\vp_j)$ of functions in $\mathcal E^1(X,\om)$ converges \textit{in energy} to $\vp \in \mathcal E^1(X,\om)$ if the convergence happens in the strong topology.\\

Let us go back to the main setting of this paper where $D$ is a smooth divisor such that $K_X+D$ is ample. Then one can choose $\om$ a Kähler form in $c_1(K_X+D)$, $\theta$ a smooth representative of $c_1(D)$, and we denote by $\omb$ $(\beta \in (0,\beta_0) say)$  the negatively curved conic Kähler-Einstein metric; it converges to the cuspidal Kähler-Einstein metric $\om_0$. The metric $\omb$ lives in the same cohomology class as $\om-\beta \theta$, and it converges weakly to $\om_0$, so one can find a family of normalized potentials $\vpb$ for $\omb$ ($\beta \in [0, \beta_0)$ still) such that $\vpb$ converges to $\vp_0$ in the $L^1$ topology. We would like to improve the weak convergence of these potentials into a strong convergence. 

One of the main issues is that $\omb$ and $\om_0$ do not live in the same cohomology class, so it is a priori not clear whether $\vpb$ is $\om$-psh. Actually, we do not know how to prove it without using Theorem B, from which it is though an obvious consequence. Indeed, we know that there exists a uniform constant $C>0$ such that $\omb$ dominates $C$ times the model conic metric, which itself dominates $\frac 1 2 \om$. Therefore, $\om + \ddc \vpb = \omb + \beta \theta \ge \frac 1 2\om + \beta \theta$ which is Kähler for $\beta$ small enough. \\

Therefore $\vpb \in \PSH(X,\om)$; as $\vpb$ happens to be bounded it follows immediately that $\vpb \in \mathcal E^1(X,\om)$. So it would make sense to study whether the weak converges is strong. This is the content of the following theorem:

\begin{theo}
The potentials $\vpb$ converge in energy toward $\vp_0$. 
\end{theo}

\begin{proof}
We first claim that there exists a uniform constant $C>0$ such that
\begin{equation}
\label{claim}
|\vpb-\vp_0| \le C -\vp_0
\end{equation}
Indeed, we proved that $\vpb= \psi_{\beta} +O(1)$, where $\psi_{\beta}=-2\log(1-\sb)/\beta$
satisfies $0\ge\psi_{\beta}\ge - 2 \log(- \log |s|^2)$ by Lemma \ref{lem:elem}. To be completely rigorous, one should add that the $\vpb$ here is a normalized version of the potential used in the previous section. But as these potentials are converging to an $\om$-psh function, their suprema admit a uniform bound, so we can ignore this detail. \\

Now, remember that we want to show that $E(\vpb) \to E(\vp_0)$; we will deal with each of the summands of the terms in the energy functional separately. For each integer $k\in [0,n]$, we have
\begin{eqnarray*}
(\om+\ddc \vpb)^k &=& (\omb+\beta \theta)^k \\
 &=& \omb^k + \sum_{j=1}^k \beta^j \theta^{j} \wedge \omb^{k-j}
\end{eqnarray*}
As $\theta$ is smooth, there exists $C>0$ such that $-C\om^j \le\theta^j \le C\om^j$, hence $0 \le \theta^j+C\om^j \le 2C \om^j$, and multiplying by $\omb^{k-j}$ we get 
$\theta^{j} \wedge\omb^{k-j} \le 2C \om^j \wedge \omb^{k-j}$. Moreover, from Theorem B, $\omb$ dominates uniformly a small multiple of $\om$, so up to increasing $C$, we get $(\om+\ddc \vpb)^k \le C \omb^{k}$. Thanks to Theorem B again, we know that $\omb \le C \om_0$, so that in the end $(\om+\ddc \vpb)^k \le C \om_0^{k}$.
Combining this with \eqref{claim}, we obtain a constant $C>0$ such that for all $k\in [0,n]$, we have the domination:
\[|\vpb-\vp_0| \, (\om+\ddc \vpb)^{k}\wedge \om^{n-k} \le C (1-\vp_0) \, \om_0^n\]

\noindent
As $\vp_0 \in L^1(\om_0)$ and $\vpb \to \vp_0$ smoothly on $X\smallsetminus D$, Lebesgue dominated convergence theorem shows that 
\[ \int_X \vpb \,  (\om+\ddc \vpb)^{k}\wedge \om^{n-k} \quad \underset{\beta\to 0}{\longrightarrow} \quad \int_X \vp_0 \, \om_0^{k}\wedge \om^{n-k} \]
which shows that $E(\vpb) \to E(\vp_0)$ when $\beta$ goes to $0$.

\end{proof}

\section{Convergence of the rescaled metrics}
In this whole paragraph, we will slightly change notation and denote by $\omb$ the Kähler-Einstein cone metric previously denoted $\omvp$.

\subsection{Cylindrical metrics}
\label{cyl}
In the next section, we will see that a suitable rescaling of $\omb$ gives rise at the limit to a very particular type of metrics on $\Cx$ that we are going to call cylindrical. 

\begin{defi}
Let $\pi:  \CC^{n} \to \Cx$ be the universal cover of $\Cx$ given by $\pi(z_1, \ldots, z_n)=(e^{z_1}, z_2, \ldots, z_n)$. A Kähler metric $\om$ on $\Cx$ is called \textit{cylindrical} if $\pi^*\om$ is isometric to $\om_{\rm eucl}$ by a complex linear transformation, i.e. there exists $g\in \Gln$ such that $\pi^*\om = g^*\ome$.
\end{defi}

These metrics are quasi-isometric to the standard cylindrical metric $\frac{idz_1\wedge d \bar z_1}{|z_1|^2}+\sum_{k\ge 2} idz_k \wedge d\bar z_k$ hence complete. Also, because $\pi$ is a local biholomorphism, any two such metrics are locally (holomorphically) isometric, but it is not obvious whether they are globally holomorphically isometric or not. We are going to investigate this question in this section. For the time being, let us try to give an explicit description of cylindrical metrics. Basically, they are just push-forward by $\pi$ of a Kähler metric on $\CC^n$ with constant coefficients. Of course the push-forward will in general produce a current, but here $\pi^{-1}(z_1, \ldots, z_n)= (\log z_1, z_2, \ldots, z_n)$ defined locally induces a globally defined form $\pi_* dz_1= (\pi^{-1})^*dz_1 = \frac{dz_1}{z_1}$. Therefore, if $\om:=\sum_{j,k} a_{j\bar k} \, idz_j \wedge d \bar z_k $ is a metric with constant coefficients, then the push-forward $\pi_*\om$ is given by
\[\pi_*\om=a_{1\bar 1} \frac{idz_1\wedge d\bar z_1}{|z_1|^2}+\sum_{j=2}^n \left( a_{1\bar j} \,i \frac{dz_1}{z_1}\wedge d\bar z_j+a_{j\bar 1} \,idz_j\wedge  \frac{d\bar z_1}{\bar z_1}\right) +\sum_{j,k\ge 2} a_{j\bar k} \, idz_j \wedge d \bar z_k \]
This is the general form of a cylindrical metric, as long as $(a_{j\bar k})_{j,k} $ is an hermitian definite positive matrix. \\

Let us now identify when two cylindrical metric are holomorphically isometric. We consider two such metrics $\om, \om'$ i.e. $\om = g^* \ome$ and $\om' = g'^*\ome$ for $g,g' \in \Gln$ and we assume that there exists $f: \Cx \to \Cx$ a biholomorphic map such that $f^*\om'=\om$. We are going to show that $f$ is necessarily of the form
\[f(z_1, \mathbf{w}) = (z_1^{\pm 1}, A \mathbf {w})\] 
for some matrix $A\in \mathrm{GL}(n-1, \CC)$ and where $\mathbf {w}=(z_2, \ldots, z_n)$. It will show that the only cylindrical metrics
holomorphically isometric to 
\[a_{1\bar 1} \frac{idz_1\wedge d\bar z_1}{|z_1|^2}+\sum_{k=2}^n \left( a_{1\bar k} \, i \frac{dz_1}{z_1}\wedge d\bar z_k+a_{k\bar 1} \, idz_k\wedge  \frac{d\bar z_1}{\bar z_1}\right) +\sum_{j,k\ge 2} a_{j\bar k} \, idz_j \wedge d \bar z_k\] 
are the ones obtained from it by the transformation $z_1\mapsto 1/z_1$ and a complex linear transformation of the $(n-1)$ variables $z_2, \ldots, z_n$. As $z_1\mapsto 1/z_1$ leaves  $\frac{dz_1\wedge d\bar z_1}{|z_1|^2}$ invariant and turns $\frac{dz_1}{z_1}\wedge d\bar z_j$ into its opposite, it acts the same way as some complex linear transformation on $z_2, \ldots, z_n$. We can go a little bit further: let us set $M=(a_{i\bar j})_{2\le i,j\le n}$, and $X={}^t  (a_{1\bar 2}, \ldots, a_{1\bar n})$. 
There exists $P\in \mathrm{GL}(n-1, \CC)$ such that $P^*MP= \mathrm{Id}$, and replacing $P$ by $UP$ for $U\in \mathrm{U}(n-1,\CC)$ preserves this property. We can choose $U$ such that $UPX=  (||PX||, 0, \ldots, 0)$, so that the holomorphic transformation $f(z_1, \mathbf{w}) = (z_1, UP \mathbf {w})$ maps the above metric to 
\[a \frac{idz_1\wedge d\bar z_1}{|z_1|^2}+\left( b \, i\frac{dz_1}{z_1}\wedge d\bar z_2+b \, i dz_2\wedge  \frac{d\bar z_1}{\bar z_1}\right) +\sum_{k\ge 2} idz_k \wedge d \bar z_k\]
where $a,b$ are positive numbers such that $a>b^2$. Moreover, any two such metrics are holomorphically isometric if and only if they have same coefficients $a$ and $b$ ($b$ being required to be positive). In particular a cylindrical metric is determined by its trace and determinant. \\

Let us now prove the claim above. By the lifting theorem for maps, the map $f \circ \pi$ can be lifted to a holomophic map $\bar f : \CC^n \to \CC^n$ sending $0$ to $0$:
$$ \xymatrix{
    \CC^n \ar[r]^{\bar f} \ar[d]_{\pi}  & \CC^n \ar[d]^{\pi}  \\
     \Cx \ar[r]^f & \Cx }$$
Moreover, $\bar f^* g'^*\ome = \bar f^*\pi^* \om'= \pi^*f^*\om' = g^*\ome$ so that $(g'\circ \bar f \circ g^{-1})^*\ome = \ome$. As a consequence, $g'\circ \bar f \circ g^{-1} \in \mathrm{U}(n,\CC)$ which implies that $\bar f\in \Gln$. 

So $\bar f=(\bar f_1, \ldots, \bar f_n)$ is a linear isomorphism of $\CC^n$ that descends to $f$. Therefore, if $x,y \in \CC^n$ satisfy $\pi(x)=\pi(y)$, we must have $\pi(\bar f(x))=\pi(\bar f(y))$. So we have both $\bar f_1(z_1+2i\pi, \mathbf{w}) = \bar f_1(z_1, \mathbf{w}) + 2ik\pi$ for some $k\in \Z$ and $\bar f_j( z_1+2ik\pi, \mathbf{w}) =\bar f_j( z_1, \mathbf{w})$
for any $j \ge 2$. This shows that 
\[\bar f=\begin{pmatrix}
k & 0 & \cdots & 0 \\
0 & & &\\
\vdots & & A& \\
0 & & & 
\end{pmatrix}\] 
for some invertible matrix $A \in \mathrm{GL}(n-1, \CC)$ which in turn proves that $f(z_1, \mathbf{w}) = (z_1^{k}, A \mathbf {w})$. As $f$ was supposed to be one-to-one, we must have $k\in \{1, -1\}$, which proves the claim. 
 
\subsection{Convergence to a cylindrical metric}
We are going to consider a small neighborhood of $D$ and rescale the metric $\omb$ there to study its asymptotic behavior near $D$. Typically, let us work in a coordinate chart $(z_1, \ldots, z_n)$ where $D=(z_1=0)$. If $\D:=\{(z_i); \forall i, |z_i| \le 1\}$, it is not hard to see that the completion of $( \D^n \smallsetminus (z_1=0), \om_{\beta, \mathrm{mod}})$ where
\[\om_{\beta, \rm mod}:= \frac{\beta^2 i dz_1 \wedge d\bar z_1}{|z_1|^{2(1-\beta)}(1-|z_1|^{2\beta})^2}+\sum_{k=2}^n i dz_k \wedge d\bar z_k \]
is given by $( \D^n,d_{\beta})$ where $d_{\beta}$ satisfies 
\[d_{\beta}(0,z) \simeq \frac 1 2 \log \left( \frac{1+|z_1|^{\beta}}{1-|z_1|^{\beta}}\right)+\sum_{k=2}^n |z_k| \]
where $\simeq$ means "is equivalent up to universal constants to".
This can be seen using the equivalence of the norms $\sum |z_k|$ and $\sqrt{\sum |z_k|^2}$ and the fact that a primitive of $\frac{\beta r}{r^{1-\beta}(1-r^{2\beta})}$ is given by 
$\frac 12 \log \left( \frac{1+r^{\beta}}{1-r^{\beta}}\right)$. 

\noindent
Therefore $B(0,r,\om_{\beta, \mathrm{mod}})$ is "equivalent" to the polydisk 
\[ \left\{(z_1, \ldots, z_n); \, |z_1|^{\beta} < \frac{1-e^{-2r}}{1+e^{-2r}} \quad \& \quad \forall k \ge 2, |z_k| < r   \right\}\]

\noindent 
These observations enable us to realize that if $p\in D$, the ball of a given radius (say $1$) centered at $p$ for the (completion of) $\omb$ is essentially given by the neighborhood of $D$ defined as $\{z\in X; \, |s(z)|^{2\beta}<e^{-1}\}$. Therefore, we set:
$$V_{\beta}:=\left\{z\in X\smallsetminus D; \, |s(z)|^{2\beta}<e^{-1}\right\}.$$
and we are going to study the convergence of $(V_{\beta}, r^{-1}_{\beta} \omb)$ for some suitable sequence $r_{\beta} \to 0$. For that kind of sets (non compact ones), an appropriate notion of convergence is the pointed Gromov-Hausdorff convergence. So we may localize the situation: given a trivializing open set $U$ for $X$ meeting $D$, we set: $U_{\beta}:=V_{\beta}\cap U$. This is a subset of $\CC^*\times \CC^{n-1}$, that we will endow with the rescaled metric $\beta^{-2}\omb$. The main result of this section is that the family $(\beta^{-2}\omb)_{0<\beta \le \beta_0}$ is relatively compact and that all its cluster values are cylindrical metrics:

\begin{theo}
\label{thm:gh}
Let $(\beta_n)_{n\in \N}$ be a sequence of positive numbers converging to $0$. Then up to extracting a subsequence, there exists a cylindrical metric $\om_{\rm cyl}$ on $\CC^*\times \CC^{n-1}$ such that the metric spaces $(U_{\beta_n}, \beta_n^{-2}\om_{\beta_n})$ converge in pointed Gromov-Hausdorff topology to $(\CC^*\times \CC^{n-1},  \om_{\rm cyl})$ when $n$ tends to $+\infty$.
\end{theo}

\noindent
Let us give a few remarks about this result:

$\--$ By this convergence mode, we mean that for any $p\in \CC^*\times \CC^{n-1}$ and any $r>0$, there exists a sequence of points $p_{n} \in U_{\beta_n}$ such that the closed balls $\bar B_{p_{n}}(r,\beta_n^{-2}\om_{\beta_n})\subset U_{\beta_n}$ converge in Gromov-Hausdorff topology to $\bar B_{p}(r,\om_{\rm cyl})\subset \CC^*\times \CC^{n-1}$.

$\--$ We are actually going to prove a much more precise result, as the convergence will be showed to happen in the $\mathscr C^{\infty}_{\rm loc}$ topology on $\CC^*\times \CC^{n-1}$, cf proof below for the precise meaning of this statement. 

$\--$ We do not know whether the family of metrics $\beta^{-2}\omb$ converges when $\beta$ tends to zero. Indeed, there could be different cluster values because of the non uniqueness of cylindrical metrics, so it does not follow from the result above. Actually, we believe that this is a difficult question.

\begin{proof}
The proof works in three main steps: first we understand the rescaling of the model metric, then we work with the family of Kähler-Einstein metrics to prove the local convergence, and finally we show that all limits are cylindrical. In the fourth step, which is very standard, we make explicit how to deduce the Gromov-Hausdorff convergence from the more precise smooth convergence. 

Let us first introduce the notation $\D(a,b) := \{(z_i) \in \CC^*\times \CC^{n-1}; |z_1|<a \, \, \mathrm{and} \, \, |z_k| <b \, \, \mathrm{for } \, k\ge 2\}$. 
In all the following, we are working on $U_{\beta}=\D(e^{-\frac{1}{2\beta}}, 1)$. \\

\noindent 
\textbf{Step 1. The model case}

$\bullet$  We first endow $U_{\beta}$ with $\beta^{-2}\ombm$ where:
\[\ombm :=\frac{\beta^2 i dz_1 \wedge d\bar z_1}{|z_1|^{2(1-\beta)}(1-|z_1|^{2\beta})^2}+\sum_{k=2}^n i dz_k \wedge d\bar z_k \]
Consider the rescaling:
\[
\begin{array}{cccc}
\Psi_{\beta}:& \D(e^{\frac{1}{2\beta}}, \beta^{-1})  & \longrightarrow &U_{\beta}= \D(e^{-\frac{1}{2\beta}}, 1)  \\
 &(w_1, w_2, \ldots, w_n)& \mapsto & (e^{-\frac{1}{\beta}}w_1, \beta w_2, \ldots, \beta w_n)
\end{array}
\]
It is a diffeomorphism, and we have 
\[\Psi_{\beta}^*(\beta^{-2} \ombm) =\frac{ e^{-2} |w_1|^{2\beta} }{(1-e^{-2}|w_1|^{2\beta})^{2}}\cdotp \frac{i dw_1 \wedge d\bar w_1}{|w_1|^2}+\sum_{k=2}^n i dw_k \wedge d\bar w_k \]

\noindent
As $|w_1|^{2\beta}$ converges to $1$ on $\CC^*$, given any compact $K \subset \CC^*\times \CC^{n-1}$, then $K\subset \D(e^{\frac{1}{2\beta}}, \beta^{-1}) $ for $\beta$ small enough (depending on $K$), and $\Psi_{\beta}^*(\beta^{-2}\ombm)$ converges to the cylindrical metric
\[\frac{ e^{-2} }{(1-e^{-2})^{2}}\cdotp \frac{i dw_1 \wedge d\bar w_1}{|w_1|^2}+\sum_{k=2}^n i dw_k \wedge d\bar w_k \]
in $\mathscr C^{\infty}(K)$.\\

$\bullet$ In terms of potentials, here is what is happening: $\ombm =\ddc \vp_{\beta, \mathrm{mod}}$ where
\[\vp_{\beta, \mathrm{mod}}(z)= -\log\left(\frac{1-|z_1|^{2\beta}}{\beta}\right)^2+\sum_{k=2}^n |z_k|^2\]
Therefore, 
\begin{equation}
\label{mod}
\Psi_{\beta}^*(\beta^{-2}\vp_{\beta, \mathrm{mod}})(w)= -\beta^{-2}\log\left( 1-e^{-2}|w_1|^{2\beta} \right)^2+\beta^{-2} \log \beta^2 + \sum_{k=2}^n |w_k|^2 
\end{equation}
Now, the expansion $|w_1|^{2\beta} = \sum_{i=0}^{+\infty} \frac{\left(\log |w_1|^2 \right)^i}{i!}\beta^i$ yields the following Taylor expansion (for $\beta \to 0$):
\begin{eqnarray*}
\Psi_{\beta}^*(\beta^{-2}\vp_{\beta, \mathrm{mod}})(w)&=&\beta^{-2} \log \beta^2 - \log(1-e^{-2})^2\beta^{-2} + a \log |w_1|^2 \beta^{-1} + \frac{a(1-a)}{2} \log^2 |w_1|^2 \\
&& +\sum_{k=2}^n |w_k|^2+O(\beta) 
\end{eqnarray*}
where $a=e^2/(1-e^{-2})$. Moreover, this series converges uniformly on compact subsets of $\CC^*\times \CC^{n-1}$. This expansion is consistent with the convergence of $\Psi_{\beta}^*(\beta^{-2}\ombm)$ obtained above as $\beta^{-2} \log \beta^2 - \log(1-e^{-2})^2\beta^{-2} + a \log |w_1|^2 \beta^{-1}$ is polyharmonic on $\CC^*\times \CC^{n-1}$ whereas $\ddc \log^2 |w_1|^2 = 2 dw_1 \wedge d \bar w_1 /|w_1|^2$.\\

\noindent 
\textbf{Step 2. Local smooth convergence}

\noindent
Let us now consider $\ombb:=\Psi_{\beta}^*(\beta^{-2}\omb)$; by Theorem B, there exists $C>0$ independent of $\beta$ such that: 
\[C^{-1} \Psi_{\beta}^*(\beta^{-2}\ombm)\le \ombb \le C \, \Psi_{\beta}^*(\beta^{-2}\ombm) \]
and therefore, given any compact set $K \subset  \CC^*\times \CC^{n-1}$, there exists a constant $C_K$ such that the comparison
\begin{equation}
\label{comp}
C_K^{-1} \ome \le \ombb \le C_K\,  \ome 
\end{equation}
is valid on $K$ (for $\beta$ small enough). 

Although $U_{\beta}$ is not simply connected, $\ombm$ admits a potential on this set, hence so does $\omb$. Therefore, $\ombb$ admits potentials on $ \D(e^{\frac{1}{2\beta}}, \beta^{-1})$.

We want to show that up to a renormalization, "the" potential of $\ombb$  and all its derivatives are uniformly bounded on $K$. It turns out that operating a sup-normalization to our globally defined potential $\vpb$ is not good enough to ensure this and that we have to subtract to $\vpb$ a carefully chosen pluriharmonic function. Let us get into the details. Thanks to \eqref{comp}, the currents $\ombb$ are uniformly bounded in mass on $K$, so by weak compactness of positive currents, they admit potentials $\vpbb$ uniformly bounded in $L^1_{\rm loc}$ norm, hence also in $L^p_{\rm loc}$ norm for any given $p>1$. Now, because of \eqref{comp} again, we know that $\Delta \vpbb$ is uniformly bounded on $K$, and therefore, so is 
\begin{equation}
\label{C0}
||\vpbb||_{L^{\infty}(K)} \le C
\end{equation}
for some uniform constant $C$ thanks to standard local properties for solutions of elliptic equations, cf \cite[Theorem 8.17]{Gilb}.\\

We will now extract all the information we can out of the Kähler-Einstein equation 
\[\Ric \ombb = -\beta^2 \ombb\]
satisfied by $\ombb=\ddc \vpbb$ on $ \D(e^{\frac{1}{2\beta}}, \beta^{-1})$. If we denote by $dV$ the euclidian volume form $dV:=\ome^n/n!$, we deduce from the previous equation that the function $\Hb:=\log\left( (\ddc \vpbb)^ne^{-\beta^2\vpb}/dV\right)$ is polyharmonic. In terms of $\vpbb$, the Kähler-Einstein equation satisfied by $\ombb$ becomes the following Monge-Ampère equation:
\begin{equation}
\label{mar}
(\ddc \vpbb)^n = e^{\beta^2 \vpbb + \Hb}dV
\end{equation}
Estimate \eqref{comp} shows that $\beta^2 \vpb + \Hb= O(1)$, hence $H_{\beta}=O(1)$ thanks to \eqref{C0}. By the standard properties of (pluri)harmonic function, we deduce:
\begin{equation}
\label{ph}
||H_{\beta}||_{\mathscr C^{k}(K)} \le C_k
\end{equation}
for some constants $C_k$ depending only on $k,K$ (and not $\beta$). \\

The next step is the $\mathscr C^{2,\alpha}$ estimate. The operator defined by $$F_{\beta}(\vp)=\log \left( \frac{(\ddc \vp)^n}{dV}\right) - \beta^2 \vp$$ is uniformly elliptic, concave as a function of $\ddc \vp$, so it is governed by Evans-Krylov theory. In particular, the $\mathscr C^{2, \alpha}$ norm of $\vp$ on $K $ is controlled by $||\vp||_{L^{\infty}(K')}$, $||\Delta \vp||_{L^{\infty}(K')}$ and $||F_{\beta}(\vp)||_{C^{0,1}(K')}$ given any compact set $K'$ containing $K$ in its interior. As $F_{\beta}(\vpbb)= H_{\beta}$ satisfies the estimate \eqref{ph} above and the solution $\vpbb$ satisfies \eqref{comp}-\eqref{C0}, we infer the existence of $\alpha\in(0,1)$ and $C>0$ such that
\[||\vpbb ||_{\mathscr C^{2,\alpha}(K)} \le C\]

From there, we deduce that the linear operator $\Delta_{\ombb}-\beta^2$ has coefficients whose $\mathscr C^{\alpha}$ norm is uniformly bounded on $K$. As this operator is the linearization of $F_{\beta}$ and as $F_{\beta}(\vpbb)=H_{\beta}$ has uniformly bounded derivatives on $K$, Schauder estimates guarantee that every derivative of $\vpbb$ is bounded on $K$ (as $K$ is arbitrary here). By Arzelà-Ascoli theorem the family $(\vpbb)_{0<\beta\le \beta_0}$ is relatively compact for the $\mathscr C^{\infty}(K)$ topology.   \\

\noindent 
\textbf{Step 3. Identification of the limit as cylindrical}

\noindent
Let $\bar \om_0$ be a cluster value of the family $(\ombb)$, realized as the limit of a sequence $\bar \omega_{\beta_n}$ where $\beta_n>0$ tends to $0$. Here, the convergence happens in $\mathscr C^{\infty}_{\rm loc}(\CC^*\times \CC^{n-1})$. By the estimate \eqref{comp} combined with the fact that $\Ric \ombb = -\beta^2 \ombb$, the metric $\bar \om_0$ would be a Ricci-flat metric quasi-isometric to $\om_{\rm cyl}$; pulling it back to $\CC^n$ by the universal cover $\pi:(z_1, \ldots, z_n) \mapsto (e^{z_1}, z_2, \ldots, z_n)$, $\pi^*\bar \om_0$ would be a Ricci-flat metric isometric to the euclidian metric. 
Therefore one would have $(\pi^*\bar \om_0)^n = e^H \om_{\rm eucl}^n$ for some bounded pluriharmonic function $H$ on $\CC^n$. So $H$ should be a constant function and by Liouville Theorem, there exists an element $g \in \mathrm{GL}(n,\CC)$ such that $\pi^*\bar \om_0=g^*\om_{\rm eucl}$. Therefore $\bar \om_0$ is a cylindrical metric. \\

\noindent 
\textbf{Step 4. From smooth to Gromov-Hausdorff convergence}

\noindent
To lighten notation, let us drop the index $n$ is this paragraph. We proved that $\ombb$ converges to a cylindrical metric $\bar \om_0$
locally smoothly. So given any $p\in \CC^*\times \CC^{n-1}$ and any $r>0$, we have $B_p(2r, \bar \om_0)\subset \D(e^{\frac{1}{2\beta}}, \beta^{-1})$ for $\beta$ small enough. So if we set $p_{\beta}=\Psi_{\beta}(p)$, the ball $B_{p_{\beta}}(r,\beta^{-2}\omb)$ is isometric to $B_p(r,\ombb)$ and the smooth convergence of $\ombb$ to $\bar \om_0$ gives both:

$\cdotp$ $\, B_p(r, \ombb) \subset B_p(2r, \bar \om_0)$ for $\beta$ small enough;

$\cdotp$ $\, B_p(r, \ombb)$ converges to $B_p(r, \bar \om_0)$ in the Gromov-Hausdorff topology.

\noindent
The second point can be see from the fact that the norm of the gradient of the identity map $(\D(e^{\frac{1}{2\beta}},\beta^{-1}), \ombb) \to (\D(e^{\frac{1}{2\beta}},\beta^{-1}), \bar \om_0)$ and its inverse both tend uniformly to $1$ on any given compact set. So this provides the expected $\ep$-isometry.  
\end{proof}

\backmatter

\bibliographystyle{smfalpha}
\bibliography{biblio}

\end{document}